\documentclass[a4paper,12pt,reqno]{amsart}
\usepackage[T1]{fontenc}
\usepackage[utf8]{inputenc}
\usepackage{bera}
\usepackage{amssymb,dsfont,mathrsfs}
\usepackage[margin=1in]{geometry}
\usepackage{enumitem}
\usepackage{graphicx,color}
\usepackage{xparse}
\usepackage[bookmarksdepth=2]{hyperref}

\numberwithin{equation}{section}

\newtheorem{theorem}{Theorem}[section]
\newtheorem{lemma}[theorem]{Lemma}
\newtheorem{corollary}[theorem]{Corollary}
\newtheorem{proposition}[theorem]{Proposition}

\theoremstyle{definition}

\newtheorem{remark}[theorem]{Remark}

\newcommand{\ind}{\mathds{1}}
\newcommand{\diag}{\delta}
\newcommand{\pr}{\mathds{P}}
\newcommand{\ex}{\mathds{E}}

\newcommand{\R}{\mathds{R}}

\newcommand{\ph}{\varphi}

\renewcommand{\le}{\leqslant}
\renewcommand{\ge}{\geqslant}

\NewDocumentCommand{\formula}{ssom}{%
 \IfBooleanTF{#1}{%
  \IfBooleanTF{#2}{%
   \IfValueTF{#3}%
    {\begin{align}\label{#3}\begin{gathered}#4\end{gathered}\end{align}}%
    {\begin{gather}#4\end{gather}}%
  }{%
   \IfValueTF{#3}%
    {\begin{align}\label{#3}\begin{aligned}#4\end{aligned}\end{align}}%
    {\begin{gather*}#4\end{gather*}}%
  }%
 }{%
  \IfValueTF{#3}%
   {\begin{align}\label{#3}#4\end{align}}%
   {\begin{align*}#4\end{align*}}%
 }%
}

\begin{document}

\title[Fleming--Viot couples live forever]{Fleming--Viot couples live forever}
\author{Mateusz Kwaśnicki}
\dedicatory{In loving memory of my father}
\thanks{Work supported by the Polish National Science Centre (NCN) grant no.\@ 2019/33/B/ST1/03098}
\address{Mateusz Kwaśnicki \\ Department of Pure Mathematics \\ Wrocław University of Science and Technology \\ ul. Wybrzeże Wyspiańskiego 27 \\ 50-370 Wrocław, Poland}
\email{mateusz.kwasnicki@pwr.edu.pl}
\date{\today}

\begin{abstract}
We prove a non-extinction result for Fleming--Viot-type systems of two particles with dynamics described by an arbitrary symmetric Hunt process under the assumption that the reference measure is finite. Additionally, we describe an invariant measure for the system, we discuss its ergodicity, and we prove that the reference measure is a stationary measure for the embedded Markov chain of positions of the surviving particle at successive branching times.
\end{abstract}

\maketitle

%
%

\section{\uppercase{Introduction}}
\label{sec:intro}

Consider a Hunt process $X(t)$ on a state space $D$, with a possibly finite lifetime~$\zeta$. The Fleming--Viot system of $K \ge 2$ particles is a collection of $K$ processes $\bar X_k(t)$, which evolve as follows. The particles are independent copies of $X(t)$ up to the random time $\tau_1$ when one of the particles reaches its lifetime. Instead of dying, however, this particle immediately jumps to the position of a randomly chosen other particle. Then the particles again evolve as independent copies of $X(t)$ until one of them reaches its lifetime, at a random time $\tau_2$, and so the process continues.

Alternatively, one can think that a randomly chosen particle branches into two whenever some particle dies. For this reason, the random times $\tau_n$ described above are often called \emph{branching times}. One expects that, under reasonable assumptions, the branching times $\tau_n$ diverge to infinity, and so the Fleming--Viot system is well-defined for every $t \in (0, \infty)$. It is known, however, that in some cases the limit $\tau_\infty = \lim_{n \to \infty} \tau_n$ can be finite with positive probability, and then there is no obvious way to extend the evolution of the Fleming--Viot system past $\tau_\infty$. Therefore, we say that \emph{extinction}, or \emph{branching explosion}, occurs at $\tau_\infty$ whenever it is finite.

The main result of this paper deals with just $K = 2$ particles. Note that in this case at each branching time there is only one surviving particle, and so both particles occupy the same location at each branching time. Our key assumption is that the Hunt processes $X(t)$ is \emph{self-dual} (or \emph{symmetric}) with respect to a \emph{finite} reference measure~$m$. Here by self-duality we simply mean that the transition operators of $X(t)$ are self-adjoint on $L^2(m)$.

\begin{theorem}
\label{thm:main}
Consider the Fleming--Viot system of two particles evolving according to a self-dual Hunt process on a state space $D$ with a finite reference measure~$m$. For almost every initial configuration of the two particles, either (a)~on $D \times D$ (with respect to the product measure $m \times m$); or (b)~on the diagonal of $D \times D$ (with respect to the measure with marginals $m$), extinction never happens.
\end{theorem}

\begin{remark}
The Fleming--Viot system of particles following the Brownian motion was introduced in~\cite{bhim,bhm}. A non-extinction result was stated already in~\cite{bhm}, but the proof given there has an error. This led to an open problem which has been resolved only for sufficiently regular domains.

More precisely, non-extinction was established rigorously in~\cite{gk} under the assumption that the domain satisfies the interior and exterior cone conditions, and the interior cone has a sufficiently large aperture. The result of~\cite{gk} also covers particles following a general diffusion with smooth coefficients. A very similar result was simultaneously proved in~\cite{bbf} for Lipschitz domains with a sufficiently small Lipschitz constant, as well as for systems of two particles in polyhedral domains.

Our Theorem~\ref{thm:main} shows that no regularity is needed for Fleming--Viot systems of two Brownian particles, and that the Brownian motion can be replaced by an arbitrary self-dual Hunt process with a finite reference measure. The question for Fleming--Viot systems of more than two particles, however, remains open.

We mention here a related stream of research~\cite{bb,be,bem,bt,tough} on the \emph{spine} (the path of the surviving particle) of the Fleming--Viot system. In particular, the results of the present paper are used in~\cite{bem} to show that for Fleming--Viot systems of two particles following the Brownian motion in an interval, the spine does not coincide with the Brownian motion conditioned to live forever.
\end{remark}

\begin{remark}
Self-duality is an essential assumption in Theorem~\ref{thm:main}. For example, if $D = (0, 1)$ and $X(t)$ is the uniform motion to the left (that is, $X(t) = X(0) - t$ for $t$ less than the life-time $\zeta = X(0)$), then $\tau_\infty = \max\{\bar X_1(0), \bar X_2(0)\}$ is always finite.

We expect that for many self-dual Hunt processes on state spaces $D$ with an infinite reference measure~$m$ the assertion Theorem~\ref{thm:main} holds true, but the assumption that $m$ is finite cannot be completely removed. In other words, if $m$ is an infinite measure, then the system may become extinct in finite time with positive probability. For example, Theorem~1.1(a) in~\cite{bbp} asserts that Bessel processes of negative dimension lead to Fleming--Viot systems of two particles with $\tau_\infty < \infty$ almost surely; see also Example~5.7 in~\cite{bbf}. We note that a Bessel process of dimension $\nu \in \R$ is a self-dual Hunt process on $(0, \infty)$, but the reference measure $r^{\nu - 1} dr$ has infinite mass near the origin when $\nu \le 0$.
\end{remark}

Under mild additional assumptions, Theorem~\ref{thm:main} holds for every initial configuration. This is illustrated by the following result, which clearly covers the case of Brownian particles.

\begin{corollary}
\label{cor:all}
Consider the Fleming--Viot system of two particles evolving according to a self-dual Hunt process $X(t)$ on a state space $D$ with a finite reference measure~$m$. Suppose that the one-dimensional distributions of $X(t)$ are absolutely continuous with respect to $m$ for every $t > 0$ and every starting point $X(0) \in D$. Then, the Fleming--Viot system of two particles never becomes extinct, regardless of the initial configuration of the particles.
\end{corollary}

Let $G(x, dy)$ be the potential kernel of $X(t)$:
\formula[eq:g]{
 G(x, A) & = \ex^x \int_0^\zeta \ind_A(X_t) dt ,
}
and denote
\formula[eq:g:norm]{
 \|G\| & = \iint\limits_{D \times D} G(x, dy) m(dx) .
}
We say that the Hunt process $X(t)$ is \emph{irreducible} if for every Borel set $A \subseteq D$ such that $m(A) > 0$ we have $G(x, A) > 0$ for almost every $x \in D$. By Theorems~6 and~29 in~\cite{schilling}, irreducibility is equivalent to the following property: for every $t > 0$ and every Borel set $A \subseteq D$ such that $m(A) > 0$ we have $\pr^x(X(t) \in A) > 0$ for almost every $x \in E$; for further equivalent definitions, we refer to~\cite{schilling}. In addition to Theorem~\ref{thm:main}, we prove the following result on ergodicity of Fleming--Viot systems of two particles.

\begin{theorem}
\label{thm:ergodic}
Consider the Fleming--Viot system of two particles $(\bar X_1(t), \bar X_2(t))$ evolving according to a self-dual Hunt process $X(t)$ on a state space $D$ with a finite reference measure~$m$. If $\|G\|$ is finite (or, more generally, if the measure $G(x, dy) m(dx)$ is $\sigma$-finite), then $G(x, dy) m(dx)$ is an invariant measure for $(\bar X_1(t), \bar X_2(t))$.

If $\|G\|$ is finite and additionally $X(t)$ is irreducible, then, for every nonnegative Borel function $\ph$ on $D \times D$ and for almost every initial configuration of the two particles (with respect to either the product measure $m \times m$ or the measure~$m$ on the diagonal of $D \times D$, as in Theorem~\ref{thm:main}), the ergodic averages of $\ph(\bar X_1(t), \bar X_2(t))$ converge with probability one:
\formula{
 \lim_{T \to \infty} \frac{1}{T} \int_0^T \ph(\bar X_1(t), \bar X_2(t)) dt & = \frac{1}{\|G\|} \iint\limits_{D \times D} \ph(x, y) G(x, dy) m(dx) .
}
\end{theorem}

The above theorem is closely related to the following result on the embedded Markov chain of branching positions.

\begin{corollary}
\label{cor:embedded}
Consider the Fleming--Viot system of two particles evolving according to a self-dual Hunt process $X(t)$ on a state space $D$ with a finite reference measure~$m$, and assume that the branching times are finite with probability one. Let $Z_n$ denote the position of the surviving particle at the $n$th branching time $\tau_n$. Then $Z_n$ is a conservative Markov chain, and $m$ is a stationary measure for $Z_n$.
\end{corollary}

\begin{remark}
Let us discuss the assumptions in Theorem~\ref{thm:ergodic} and Corollary~\ref{cor:embedded}. By definition, $\|G\| = \int_D \ex^x \, \zeta m(dx)$. Therefore, if $\|G\|$ is finite, then $\pr^x(\zeta < \infty) = 1$ for almost every $x \in D$, and hence the branching times $\tau_n$ are finite with probability one for almost every initial configuration (with respect to both the product measure $m \times m$ and the measure on the diagonal with marginals $m$). Converse implication need not be true, but if $\pr^x(\zeta < \infty) = 1$ for almost every $x \in D$, then $G(x, dy) m(dx)$ is a $\sigma$-finite measure; see, for example, items~(iv) and~(v) of Proposition~(2.2) in~\cite{getoor}.

Let us define $D_0 = \{x \in D : \pr^x(\zeta < \infty) = 1\}$. As it was kindly pointed out by the referee, for a self-dual Hunt process $X(t)$ on a state space $D$ with a finite reference measure $m$, it can be proved that $D_0$ is an invariant set, and $\pr^x(\zeta = \infty) = 1$ for almost every $x \in D \setminus D_0$ (see Remark~\ref{rem:dissipative}). From the point of view of Fleming--Viot systems of particles, we can safely ignore $D \setminus D_0$, or, in other words, we may restrict our attention to the case when $\pr^x(\zeta < \infty) = 1$ for almost every $x \in D$. Thus, the assumptions in the first part of Theorem~\ref{thm:ergodic} and in Corollary~\ref{cor:embedded} are quite natural and not restrictive. We refer to~\cite{getoor-book} for further discussion.
\end{remark}

The remaining part of the paper is divided into four sections. In Section~\ref{sec:idea} we illustrate the main idea of the proof when $X(t)$ is the standard Brownian motion in a smooth domain in $\R^d$. In Section~\ref{sec:bivariate} we consider two independent copies of the underlying Hunt process $X(t)$. This part contains two auxiliary lemmas, which describe the location of the surviving particle when the other one dies. Section~\ref{sec:fv} contains the proof of Theorem~\ref{thm:main}, while in Section~\ref{sec:ergodic} we prove Theorem~\ref{thm:ergodic}.

%
%

\section{\uppercase{Idea of the proof}}
\label{sec:idea}

Let us consider a Fleming--Viot system of two particles, that from now on we denote by $(\bar X(t), \bar Y(t))$ rather than $(\bar X_1(t), \bar X_2(t))$. In this section we assume that $\bar X(t)$ and $\bar Y(t)$ evolve as independent Brownian motions in a smooth, bounded Euclidean domain $D$, and whenever one of the particles hits the boundary, it immediately jumps to the position of the other particle. The key idea of our proof lies in the following observation. Its refined variant, Lemma~\ref{lem:dynkin}, is proved rigorously in Section~\ref{sec:bivariate}.

\begin{proposition}
\label{prop:dynkin}
If the Fleming--Viot pair of particles is started at a random point distributed uniformly over the diagonal of $D \times D$, then at the first branching time the particles are again distributed uniformly on the diagonal of $D \times D$.
\end{proposition}

With the above result at hand, the remaining part of the proof of Theorem~\ref{thm:main} is fairly simple. The sequence $(Z_n, \sigma_n)$ of branching locations $Z_n = \bar X(\tau_n) = \bar Y(\tau_n)$ and gaps between branching times $\sigma_n = \tau_n - \tau_{n - 1}$ forms a stationary Markov chain, and hence, by the ergodic theorem, $\tau_n / n$ converges to a positive limit with probability one. In particular, $\tau_\infty = \lim_{n \to \infty} \tau_n$ is necessarily infinite, that is, there is no extinction.

We remark that our proof of Theorem~\ref{thm:main} in Section~\ref{sec:fv} uses the full statement of Lemma~\ref{lem:dynkin} rather than just Proposition~\ref{prop:dynkin}, and it is in fact even simpler than the argument sketched above.

Our proof of Proposition~\ref{prop:dynkin} (or Lemma~\ref{lem:dynkin}) in the next section uses relatively standard, but rather abstract tools from the theory of Markov processes. For this reason, we think it will be instructive to discuss first a more direct approach to Proposition~\ref{prop:dynkin} for Brownian particles in a smooth domain $D$. While the two arguments (the one that follows and the actual proof in Section~\ref{sec:bivariate}) may appear quite different, in fact both of them follow the same line: we identify the occupation density of the bivariate process with the Green function of $D$ (using an analytic approach in Step~1 below, and a probabilistic reasoning in Lemma~\ref{lem:reduction}), and then we determine the distribution at the first branching time (expressing it in terms of the Poisson kernel of $D$ in Step~2 below, and using resolvent techniques in Lemma~\ref{lem:dynkin}).

\begin{proof}[Sketch of the proof of Proposition~\ref{prop:dynkin}]
We divide the argument into three steps.

\emph{Step 1.} Let $G_D(x, y)$ denote the usual Green function in $D$. That is, we have $G_D(x, y) = G_D(y, x)$, for every $y \in D$ the function $x \mapsto G_D(x, y)$ vanishes continuously on the boundary of $D$, and we have
\formula{
 \Delta_x G_D(x, y) & = -\delta_y(dx)
}
in the sense of distributions. Here $\Delta_x$ stands for the Laplace operator in $\R^d$ acting on the variable $x$, and $\delta_y$ is the Dirac delta at $y$. By symmetry,
\formula{
 \Delta_y G_D(x, y) & = -\delta_x(dy) .
}
It follows that, as a bivariate function, the Green function satisfies
\formula{
 \Delta_{x_1, x_2} G_D(x_1, x_2) & = -2 \diag(dx_1, dx_2) .
}
Here and below we denote by $\diag$ the uniform measure on the diagonal of $D \times D$:
\formula{
 \diag(dx, dy) & = \delta_x(dy) dx = \delta_y(dx) dy .
}
Since $G_D(x_1, x_2)$ converges to zero on the boundary of $D \times D$, Green's third identity implies that
\formula{
 G_D(x_1, x_2) & = 2 \iint\limits_{D \times D} G_{D \times D}((x_1, x_2), (y_1, y_2)) \diag(dy_1, dy_2) \\
 & = 2 \int_D G_{D \times D}((x_1, x_2), (y, y)) dy .
}
Here $G_{D \times D}((x_1, x_2), (y_1, y_2))$ denotes the Green function in $D \times D$.

\emph{Step 2.} Consider the $2d$-dimensional Brownian motion $(X(t), Y(t))$ in $D \times D$, absorbed at the boundary, and denote by $\tau$ the hitting time of the boundary. By Kakutani's formula for the solution of the Dirichlet problem, if $(X(0), Y(0)) = (x_1, x_2)$, then the density function of the distribution of $(X(\tau), Y(\tau))$ on the boundary of $D \times D$ is the Poisson kernel $P_{D \times D}((x_1, x_2), (z_1, z_2))$. Recall that the Poisson kernel is the inward normal derivative of the Green function (with respect to the second variable) on the boundary $\partial(D \times D) = (\partial D \times D) \cup (D \times \partial D)$:
\formula{
 P_{D \times D}((x_1, x_2), (z_1, z_2)) & = \begin{cases} \dfrac{\partial}{\partial \nu} \biggr|_{z_1} G_{D \times D}((x_1, x_2), (\cdot, z_2)) & \text{for $(z_1, z_2) \in \partial D \times D$,} \\[1em] \dfrac{\partial}{\partial \nu} \biggr|_{z_2} G_{D \times D}((x_1, x_2), (z_1, \cdot)) & \text{for $(z_1, z_2) \in D \times \partial D$,} \end{cases}
}
Here $\tfrac{\partial}{\partial \nu}\rvert_z$ denotes the inward normal derivative on the boundary of $D$ evaluated at a point $z \in \partial D$.

It follows that if the initial position $(X(0), Y(0))$ is uniformly distributed over the diagonal of $D \times D$, then the density function of the distribution of $(X(\tau), Y(\tau))$ is given by
\formula{
 P(z_1, z_2) & = \frac{1}{|D|} \int_D P_{D \times D}((x, x), (z_1, z_2)) dx ,
}
and hence it is the inward normal derivative on the boundary of $D \times D$ of
\formula{
 \frac{1}{|D|} \int_D G_{D \times D}((x, x), (y_1, y_2)) dx & = \frac{1}{2 |D|} \, G_D(y_1, y_2) .
}
In other words,
\formula{
 P(z_1, z_2) & = \begin{cases} \dfrac{\partial}{\partial \nu} \biggr|_{z_1} \dfrac{G_D(\cdot, z_2)}{2 |D|} & \text{for $(z_1, z_2) \in \partial D \times D$,} \\[1em] \dfrac{\partial}{\partial \nu} \biggr|_{z_2} \dfrac{G_D(z_1, \cdot)}{2 |D|} & \text{for $(z_1, z_2) \in D \times \partial D$.} \end{cases}
}
But the inward normal derivative (with respect to $y$) of the Green function in $D$, $G_D(x, y)$, is the Poisson kernel in $D$, $P_D(x, z)$. Thus,
\formula{
 P(z_1, z_2) & = \begin{cases} \dfrac{P_D(z_1, z_2)}{2 |D|} & \text{for $(z_1, z_2) \in D \times \partial D$,} \\[1em] \dfrac{P_D(z_2, z_1)}{2 |D|} & \text{for $(z_1, z_2) \in \partial D \times D$.} \end{cases}
}
Suppose that $Z = X(\tau)$ if $Y(\tau) \in \partial D$, and $Z = Y(\tau)$ if $X(\tau) \in \partial D$. What we have found above implies that for every Borel set $A \subseteq D$,
\formula{
 \pr(Z \in A) & = \pr(X(\tau) \in A , \, Y(\tau) \in \partial D) + \pr(X(\tau) \in \partial D , \, Y(\tau) \in A) \\
 & = \iint\limits_{A \times \partial D} P(z_1, z_2) dz_1 dz_2 + \iint\limits_{\partial D \times A} P(z_1, z_2) dz_1 dz_2 \\
 & = \frac{1}{2 |D|} \iint\limits_{A \times \partial D} P_D(z_1, z_2) dz_1 dz_2 + \frac{1}{2 |D|} \iint\limits_{\partial D \times A} P_D(z_2, z_1) dz_1 dz_2 .
}
However, the integral of $P_D(x, y)$ over $y \in \partial D$ is equal to one, and so
\formula{
 \pr(Z \in A) & = \frac{1}{2 |Z|} \int_A 1 dz_1 + \frac{1}{2 |D|} \int_A 1 dz_2 = \frac{|A|}{|D|} \, ,
}
that is, $Z$ is uniformly distributed over $D$.

\emph{Step 3.} Let us now consider the Fleming--Viot pair $\bar X(t)$ and $\bar Y(t)$ evolving as independent Brownian motions in $D$, but whenever either particle hits the boundary of $D$, it is immediately resurrected at the location of the other particle. We suppose that $\bar X(0) = \bar Y(0)$ is uniformly distributed over $D$, and we claim that in this case also the first branching location $Z_1 = \bar X(\tau_1) = \bar Y(\tau_1)$ is uniformly distributed over~$D$.

Observe that up to time $\tau_1$, $(\bar X(t), \bar Y(t))$ is the Brownian motion in $D \times D$, and $Z_1 = \bar X(\tau_1) = \bar Y(\tau_1)$ is equal to the left limit $\bar X(\tau_1-) = \bar X(\tau_1)$ if $\bar Y(\tau_1-) \in \partial D$, and to the left limit $\bar Y(\tau_1-) = \bar Y(\tau_1)$ if $\bar X(\tau_1-) \in \partial D$. Therefore, $Z_1$ is the same random variable as the variable $Z$ studied in the previous step. This completes the proof: we have already shown that $Z$ is uniformly distributed over $D$.
\end{proof}

%
%

\section{\uppercase{Where were you}}
\label{sec:bivariate}

In this section we prove an auxiliary result, which describes the distribution of the position of the Hunt process $X(t)$ at the lifetime of its independent copy. This turns out to be the key ingredient of the proof of Theorem~\ref{thm:main} in the next section.

Suppose that $X(t)$ is a self-dual Hunt processes with state space $D \cup \{\partial\}$, where $\partial$ is a cemetery point, and a finite reference measure $m$. By $\mathscr F_t$ we denote the natural filtration of $X(t)$. For notational convenience, with no loss of generality we assume that $m$ is a probability measure. As it is customary, we adopt the convention that $X(\infty) = \partial$, and that if $f$ is a function on $D$, then we automatically extend $f$ to $D \cup \{\partial\}$ so that $f(\partial) = 0$. We denote by $\ind$ the constant one on $D$; note, however, that $\ind(\partial) = 0$. We also write $\langle f, g \rangle$ for the inner product of $f, g \in L^2(m)$. Finally, we denote by $\diag$ the measure concentrated on the diagonal of $D \times D$ with marginals $m$,
\formula{
 \diag(dx, dy) & = \delta_x(dy) m(dx) = \delta_y(dx) m(dy) .
}

We write $\pr^x$ for the probability corresponding to the process $X(t)$ started at $X(0) = x$. We denote by $\zeta$ the lifetime of $X(t)$, by
\formula{
 p_t(x, A) & = \pr^x(X(t) \in A)
}
the transition kernel of $X(t)$, and by
\formula{
 P_t f(x) & = \ex^x f(X(t)) = \int_D f(y) p_t(x, dy)
}
the transition operators $P_t$ of the process $X(t)$, defined whenever the expectation or the integral is well-defined. Then $P_t$ form a strongly continuous semigroup of self-adjoint contractions on $L^2(D, m)$. We also define the resolvent kernel
\formula{
 u_\lambda(x, A) & = \int_0^\infty e^{-\lambda t} p_t(x, A) dt = \ex^x \int_0^\zeta e^{-\lambda t} \ind_A(X(t)) dt ,
}
and the resolvent operators
\formula{
 U_\lambda f(x) & = \int_0^\infty e^{-\lambda t} P_t f(x) dt = \ex^x \int_0^\zeta e^{-\lambda t} f(X(t)) dt ,
}
defined whenever $\lambda \ge 0$ and the double integral on the right-hand side is well-defined. Note that for $\lambda = 0$ we recover the potential kernel $G(x, dy) = u_0(x, dy)$ discussed in the introduction. For further information about Hunt processes and their transition and resolvent operators, we refer to~\cite{cw}.

Suppose that $Y(t)$ an independent copy of $X(t)$. Let us write $\pr^{x, y}$ for the probability corresponding to processes $X(t)$ and $Y(t)$ started at $X(0) = x$ and $Y(0) = y$, and $\pr^\diag$ for the probability corresponding to processes $X(t)$ and $Y(t)$ started at a random point $X(0) = Y(0)$ with distribution $m$:
\formula{
 \ex^\diag \ph(X(t), Y(t)) & = \int_D \ex^{x, x} \ph(X(t), Y(t)) m(dx) .
}
We stress that under $\pr^\diag$, the processes $X(t)$ and $Y(t)$ are not independent. Denote by $\zeta^X$, $\zeta^Y$ the lifetimes of $X(t)$ and $Y(t)$, respectively. We view the bivariate process $(X(t), Y(t))$ as a Hunt process on state space $(D \cup \{\partial\}) \times (D \cup \{\partial\})$, with lifetime $\max\{\zeta^X, \zeta^Y\}$, and cemetery state $(\partial, \partial)$. Clearly, for every Borel sets $A, B \subseteq D \cup \{\partial\}$, for every Borel functions $f, g$ on $D \cup \{\partial\}$, and for every $x, y \in D$,
\formula*{
 \pr^{x, y} \bigl(X(t) \in A, \, Y(t) \in B\bigr) = p_t(x, A) p_t(y, B) , \\
 \ex^{x, y} \bigl(f(X(t)) g(Y(t))\bigr) = P_t f(x) P_t f(y) .
}
Here we abuse the notation by setting $p_t(x, \{\partial\}) = 1 - p_t(x, D)$.

Symmetry (or self-duality) of $X(t)$ allows us to link the resolvent kernels of the bivariate process $(X(t), Y(t))$ and the original process $X(t)$.

\begin{lemma}
\label{lem:reduction}
For bounded Borel functions $f$ and $g$ and $\lambda > 0$, or for nonnegative Borel functions $f$ and $g$ and $\lambda \ge 0$, we have
\formula[eq:reduction]{
 \ex^\diag \int_0^\infty e^{-2 \lambda t} f(X(t)) g(Y(t)) dt & = \frac{1}{2} \, \langle f, U_\lambda g \rangle .
}
\end{lemma}

Note that since we agreed that $f(\partial) = g(\partial) = 0$, we have $f(X(t)) g(Y(t)) = 0$ when $t \ge \min\{\zeta^X, \zeta^Y\}$. Thus, the integral in~\eqref{eq:reduction} is effectively over $t \in [0, \min\{\zeta^X, \zeta^Y\})$.

\begin{proof}
Clearly,
\formula{
 \ex^\diag \bigl(f(X(t)) g(Y(t))\bigr) & = \int_D \ex^{x, x} \bigl(f(X(t)) g(X(t))\bigr) m(dx) \\
 & = \int_D P_t f(x) P_t g(x) m(dx) \\
 & = \langle P_t f, P_t g \rangle .
}
Using the fact that $P_t$ is self-adjoint and the semigroup property, we find that
\formula[eq:duality]{
 \ex^\diag \bigl(f(X(t)) g(Y(t))\bigr) & = \langle f, P_t P_t g \rangle = \langle f, P_{2 t} g \rangle .
}
Thus, by Fubini's theorem,
\formula{
 \ex^\diag \biggl(\int_0^\infty e^{-2 \lambda t} f(X(t)) g(Y(t)) dt\biggr) & = \int_0^\infty \langle f, e^{-2 \lambda t} P_{2 t} g \rangle dt \\
 & = \biggl\langle f, \int_0^\infty e^{-2 \lambda t} P_{2 t} g dt \biggr\rangle \\
 & = \langle f, \tfrac{1}{2} U_\lambda g \rangle ,
}
as desired.
\end{proof}

Our next result is a refined version of Proposition~\ref{prop:dynkin}. Before we state it, we recall that if $\zeta^X \le \zeta^Y$, then we understand that $X(\zeta^Y) = \partial$ and so $f(X(\zeta^Y)) = 0$. We also agree that $e^{-\infty} = 0$.

\begin{lemma}
\label{lem:dynkin}
If $f$ is a bounded Borel function and $\lambda > 0$, then
\formula*[eq:dynkin]{
 \ex^\diag \bigl(e^{-2 \lambda \zeta^Y} f(X(\zeta^Y))\bigr) & = \frac{1}{2} \, \ex^\diag \bigl(e^{-\lambda \zeta^X} f(X(0))\bigr) \\
 & = \frac{1}{2} \int_D f(x) \ex^x e^{-\lambda \zeta^X} m(dx) .
}
When $\lambda = 0$ and $\pr^x(\zeta^X < \infty) = 1$ for almost all $x \in D$, then we have
\formula[eq:dynkin:zero]{
 \ex^\diag f(X(\zeta^Y)) & = \frac{1}{2} \int_D f(x) m(dx) .
}
\end{lemma}

\begin{remark}
\label{rem:dissipative}
Self-duality of $X(t)$ and finiteness of $m$ imply that $\pr^x(\zeta^X < \infty) \in \{0, 1\}$ for almost every $x \in D$. Indeed: if $f(x) = \pr^x(\zeta^X < \infty) = \lim_{t \to \infty} P_t \ind(x)$, then, arguing as in~\eqref{eq:duality}, we have
\formula{
 \int_D (f(x))^2 m(dx) & = \lim_{t \to \infty} \int_D (P_t \ind(x))^2 m(dx) \\
 & = \lim_{t \to \infty} \langle P_t \ind, P_t \ind \rangle \displaybreak[0] \\
 & = \lim_{t \to \infty} \langle \ind, P_{2 t} \ind \rangle \displaybreak[0] \\
 & = \lim_{t \to \infty} \int_D P_{2 t} \ind(x) m(dx) \\
 & = \int_D f(x) m(dx) .
}
Therefore, $f(x) - (f(x))^2$ is a nonnegative function with integral zero. This is only possible if $f(x) \in \{0, 1\}$ for almost every $x \in D$, as claimed. Additionally, $f = P_t f$ for every $t > 0$, and hence the set of $x \in D$ such that $\pr^x(\zeta^X < \infty) = 1$ is an invariant set. Thus, if $X(t)$ is irreducible, then either $X(t)$ is conservative (in the sense that $\pr^x(\zeta^X = \infty) = 1$ for almost every $x \in D$) or $\pr^x(\zeta^X < \infty) = 1$ for almost every $x \in D$. We refer to~\cite{schilling} for a closely related discussion.
\end{remark}

\begin{remark}
One can prove that $X(t)$ has an \emph{exit law} $\ell_t(x)$ such that if $s, t > 0$, then $P_s \ell_t(x) = \ell_{t + s}(x)$ for almost every $x \in D$, and $\pr^x(\zeta^X \in dt) = \ell_t(x) dt$ on $(0, \infty)$ for almost every $x \in D$. Thus, using independence of $X(t)$ and $\zeta^Y$, we have
\formula{
 \ex^\diag \bigl(e^{-2 \lambda \zeta^Y} f(X(\zeta^Y))\bigr) & = \int_D \ex^{x, x} \bigl(e^{-2 \lambda \zeta^Y} f(X(\zeta^Y))\bigr) m(dx) \\
 & = \int_D \ex^x \bigl(e^{-2 \lambda t} f(X(t)) \ell_t(x) \bigr) m(dx) \displaybreak[0] \\
 & = \int_D \int_0^\infty e^{-2 \lambda t} P_t f(x) \ell_t(x) dt m(dx) \\
 & = \int_0^\infty e^{-2 \lambda t} \langle P_t f, \ell_t\rangle dt .
}
Since $P_t$ is a self-adjoint operator,
\formula{
 \ex^\diag \bigl(e^{-2 \lambda \zeta^Y} f(X(\zeta^Y))\bigr) & = \int_0^\infty e^{-2 \lambda t} \langle f, P_t \ell_t\rangle dt \\
 & = \int_0^\infty e^{-2 \lambda t} \langle f, \ell_{2 t}\rangle dt \\
 & = \frac{1}{2} \int_0^\infty e^{-\lambda s} \langle f, \ell_s\rangle ds .
}
Undoing the initial steps, we conclude that
\formula{
 \ex^\diag \bigl(e^{-2 \lambda \zeta^Y} f(X(\zeta^Y))\bigr) & = \frac{1}{2} \int_D \int_0^\infty e^{-\lambda s} f(x) \ell_s(x) ds m(dx) \\
 & = \frac{1}{2} \int_D f(x) \ex^x e^{-\lambda \zeta^X} m(dx) ,
}
and Lemma~\ref{lem:dynkin} follows. The above argument, kindly suggested by the referee, is very elegant and insightful, but it depends on the existence of the exit law $\ell_t(x)$. This fact can be proved using Proposition~(3.7) in~\cite{fg} and an appropriate approximation argument, similar to the one given below; see also~\cite{getoor-book}. With some effort, one can also extend Lemma~\ref{lem:dynkin} to the case of a $\sigma$-finite reference measure $m$. However, we choose to restrict our attention to finite reference measures $m$, and we give a slightly longer, but more elementary proof.
\end{remark}

\begin{proof}
The idea of the proof is to apply Lemma~\ref{lem:reduction} to $f$ and $L \ind$, where $L$ denotes the generator of the process $X(t)$. However, $\ind$ typically fails to be in the domain of $L$. In order to circumvent this difficulty, we use resolvent techniques. More precisely, we approximate the (non-existent) function $L \ind$ by $\mu \ind - \mu^2 U_\mu \ind$, where $\mu \to \infty$.

Let $\mu > \lambda > 0$, and let $g = \mu \ind - \mu^2 U_\mu \ind$. Our goal is to apply Lemma~\ref{lem:reduction} to $f$ and $g$, simplify both sides of~\eqref{eq:reduction}, and pass to the limit as $\mu \to \infty$. We divide the proof into three steps.

\emph{Step 1.} By the resolvent equation,
\formula{
 U_\lambda g & = \mu U_\lambda \ind - \mu^2 U_\lambda U_\mu \ind \\
 & = \mu U_\lambda \ind + \frac{\mu^2}{\mu - \lambda} \, (U_\mu \ind - U_\lambda \ind) \\
 & = \frac{\mu^2}{\mu - \lambda} \, U_\mu \ind - \frac{\mu \lambda}{\mu - \lambda} \, U_\lambda \ind .
}
As $\mu \to \infty$, we have $\mu U_\mu \ind \to \ind$ in $L^2(D, m)$, and hence $U_\lambda g \to \ind - \lambda U_\lambda \ind$ in $L^2(D, m)$. It follows that
\formula{
 \lim_{\mu \to \infty} \langle f, U_\lambda g \rangle & = \langle f, \ind - \lambda U_\lambda \ind \rangle \\
 & = \int_D f(x) (1 - \lambda U_\lambda \ind(x)) m(dx).
}
Finally, since $\ind(X(t)) = 1$ when $t < \zeta^X$ and $\ind(X(t)) = 0$ for $t \ge \zeta^X$, we have
\formula{
 1 - \lambda U_\lambda \ind(x) & = 1 - \ex^x \int_0^\infty \lambda e^{-\lambda t} \ind(X(t)) dt \\
 & = 1 - \ex^x \int_0^{\zeta^X} \lambda e^{-\lambda t} dt = \ex^x e^{-\lambda \zeta^X} .
}
Therefore,
\formula*[eq:dynkin:a]{
 \lim_{\mu \to \infty} \langle f, U_\lambda g \rangle & = \int_D f(x) \ex^x e^{-\lambda \zeta^X} m(dx) = \ex^\diag \bigl(e^{-\lambda \zeta^X} f(X(0))\bigr) .
}
The above identity links the right-hand side of~\eqref{eq:reduction} and the right-hand side of~\eqref{eq:dynkin}.

\emph{Step 2.} For $x, y \in D \cup \{\partial\}$ we have
\formula{
 g(y) & = \mu \ind(y) - \mu^2 U_\mu \ind(y) \\
 & = \ex^{x, y} \int_0^\infty \mu^2 e^{-\mu s} \bigl(\ind(Y(0)) - \ind(Y(s))\bigr) ds .
}
Using the above formula and the Markov property (with the usual abuse of notation), alongside with Fubini's theorem, we find that
\formula{
 \hspace*{3em} & \hspace*{-3em} \ex^\diag \biggl(\int_0^\infty e^{-2 \lambda t} f(X(t)) g(Y(t)) dt\biggr) \\
 & = \int_0^\infty e^{-2 \lambda t} \ex^\diag \bigl(f(X(t)) g(Y(t))\bigr) dt \displaybreak[0] \\
 & = \int_0^\infty e^{-2 \lambda t} \ex^\diag \biggl( f(X(t)) \ex^{X(t), Y(t)} \int_0^\infty \mu^2 e^{-\mu s} \bigl(\ind(Y(s)) - \ind(Y(0))\bigr) ds \biggr) dt \displaybreak[0] \\
 & = \int_0^\infty e^{-2 \lambda t} \ex^\diag \biggl( f(X(t)) \ex^\diag \biggl( \int_0^\infty \mu^2 e^{-\mu s} \bigl(\ind(Y(t + s)) - \ind(Y(t))\bigr) ds \bigg| \mathscr F_t \biggr)\biggr) dt \displaybreak[0] \\
 & = \int_0^\infty e^{-2 \lambda t} \ex^\diag \biggl( f(X(t)) \int_0^\infty \mu^2 e^{-\mu s} \bigl(\ind(Y(t + s)) - \ind(Y(t))\bigr) ds \biggr) dt \\
 & = \ex^\diag \biggl(\int_0^\infty \int_0^\infty \mu^2 e^{-\mu s} e^{-2 \lambda t} f(X(t)) \bigl(\ind(Y(t + s)) - \ind(Y(t))\bigr) ds dt\biggr) .
}
Recall that $\ind(Y(t)) = 1$ if $t < \zeta^Y$ and $\ind(Y(t)) = 0$ otherwise. It follows that
\formula{
 \hspace*{3em} & \hspace{-3em} \ex^\diag \biggl(\int_0^\infty e^{-2 \lambda t} f(X(t)) g(Y(t)) dt\biggr) \\
 & = \ex^\diag \biggl(\int_0^\infty \int_0^\infty \mu^2 e^{-\mu s} e^{-2 \lambda t} f(X(t)) \ind_{\{t < \zeta^Y \le t + s\}} ds dt\biggr) .
}
By Fubini's theorem,
\formula{
 \hspace*{3em} & \hspace*{-3em} \ex^\diag \biggl(\int_0^\infty e^{-2 \lambda t} f(X(t)) g(Y(t)) dt\biggr) \\
 & = \ex^\diag \biggl(\int_0^\infty \int_0^\infty \mu^2 e^{-\mu s} e^{-2 \lambda t} f(X(t)) \ind_{\{t < \zeta^Y \le t + s\}} dt ds\biggr) \\
 & = \ex^\diag \biggl(\int_0^\infty \int_0^{\zeta^Y} \mu^2 e^{-\mu s} e^{-2 \lambda t} f(X(t)) \ind_{\{\zeta^Y \le t + s\}} dt ds\biggr) .
}
Substituting $s = u / \mu$ and $t = \zeta^Y - v / \mu$, we arrive at
\formula{
 \hspace*{3em} & \hspace*{-3em} \ex^\diag \biggl(\int_0^\infty e^{-2 \lambda t} f(X(t)) g(Y(t)) dt\biggr) \\
 & = \ex^\diag \biggl(\int_0^\infty \int_0^{\mu \zeta^Y} e^{-u} e^{-2 \lambda (\zeta^Y - v / \mu)} f(X(\zeta^Y - \tfrac{v}{\mu})) \ind_{\{v \le u \}} \ind_{\{\zeta^Y < \infty\}} dv du \biggr) .
}
Since $f$ is bounded and $e^{-2 \lambda (\zeta^Y - v / \mu)} \le 1$, the dominated convergence theorem applies, and we conclude that
\formula*[eq:dynkin:b]{
 \hspace*{3em} & \hspace*{-3em} \lim_{\mu \to \infty} \ex^\diag \biggl(\int_0^\infty e^{-2 \lambda t} f(X(t)) g(Y(t)) dt\biggr) \\
 & = \ex^\diag \biggl(\int_0^\infty e^{-u} \int_0^\infty e^{-2 \lambda \zeta^Y} f(X(\zeta^Y)) \ind_{\{v \le u \}} \ind_{\{\zeta^Y < \infty\}} dv du\biggr) \\
 & = \ex^\diag \bigl(e^{-2 \lambda \zeta^Y} f(X(\zeta^Y))\bigr) .
}
This identity provides a link between the left-hand sides of~\eqref{eq:reduction} and~\eqref{eq:dynkin}.

\emph{Step 3.} The desired result for $\lambda > 0$ follows now from~\eqref{eq:dynkin:a} and~\eqref{eq:dynkin:b}, combined with Lemma~\ref{lem:reduction}. Finally, the result for $\lambda = 0$ is shown by passing to the limit as $\lambda \to 0^+$ and using the dominated convergence theorem.
\end{proof}

We will need the following simple property:
\formula[eq:zero]{
 \pr^\diag(\zeta^X = \zeta^Y < \infty) & = 0 .
}
This follows easily from the fact that the semigroup $P_t$ is analytic (see Theorem~1 in Section~III.1, p.~67, in~\cite{stein}). Indeed: the function $\pr^\diag(\zeta^X > s, \zeta^Y > t) = \langle P_s \ind, P_t \ind \rangle$ is real-analytic with respect to $s, t > 0$, and so the joint distribution of $(\zeta^X, \zeta^Y)$ under $\pr^\diag$ is absolutely continuous on $(0, \infty) \times (0, \infty)$ (with a real-analytic density function). Alternatively, one can derive~\eqref{eq:zero} from formula~\eqref{eq:dynkin:zero} with $f = \ind$, with a minor twist when $\pr^\diag(\zeta^X = \infty) \in (0, 1)$ (which is only possible when $X(t)$ is not irreducible). For a closely related result, see Proposition~(3.7) in~\cite{fg} or Proposition~(6.20)(i) in~\cite{gs}.

We define the random time
\formula{
 \sigma & = \min\{\zeta^X, \zeta^Y\} ,
}
and the random variable
\formula[eq:z]{
 Z & = \begin{cases} X(\sigma) = X(\zeta^Y) & \text{if } \zeta^Y < \zeta^X , \\ Y(\sigma) = Y(\zeta^X) & \text{if } \zeta^X < \zeta^Y , \\ \partial & \text{if } \zeta^X = \zeta^Y . \end{cases} 
}
Recall that $f(X(t)) g(Y(t)) = 0$ when $t \ge \sigma$. Thus, Lemma~\ref{lem:reduction} reads
\formula[eq:reduction:sigma]{
 \ex^\diag \int_0^\sigma e^{-2 \lambda t} f(X(t)) g(Y(t)) dt & = \frac{1}{2} \, \langle f, U_\lambda g \rangle .
}
When $\lambda > 0$ and $f = g = \ind$, we obtain
\formula{
 \ex^\diag (1 - e^{-2 \lambda \sigma}) & = \lambda \langle \ind, U_\lambda \ind \rangle .
}
On the other hand, setting $\lambda = 0$ and $f = g = \ind$ leads to
\formula{
 \ex^\diag \sigma & = \tfrac{1}{2} \langle \ind, U_0 \ind \rangle = \tfrac{1}{2} \|G\|
}
(see~\eqref{eq:g:norm}). Thus, the assumption that $\|G\|$ is finite in Theorem~\ref{thm:ergodic} is equivalent to finiteness of $\ex^\diag \sigma$.

We now rephrase Lemma~\ref{lem:dynkin}. Note that $e^{-2 \lambda \zeta^Y} f(X(\zeta^Y)) = e^{-2 \lambda \sigma} f(X(\sigma))$, as $\zeta^Y = \sigma$ when $\zeta^Y < \zeta^X$, and both sides are equal to zero otherwise. Thus, formula~\eqref{eq:dynkin} can be written as
\formula{
 \ex^\diag \bigl(e^{-2 \lambda \sigma} f(X(\sigma))\bigr) & = \tfrac{1}{2} \ex^\diag \bigl(e^{-\lambda \zeta^X} f(X(0))\bigr) .
}
By symmetry,
\formula{
 \ex^\diag \bigl(e^{-2 \lambda \sigma} f(Y(\sigma))\bigr) & = \tfrac{1}{2} \ex^\diag \bigl(e^{-\lambda \zeta^X} f(X(0))\bigr) ,
}
and therefore
\formula[eq:key]{
 \ex^\diag \bigl(e^{-2 \lambda \sigma} f(Z)\bigr) & = \ex^\diag \bigl(e^{-\lambda \zeta^X} f(X(0))\bigr) = \int_D f(x) \ex^x e^{-\lambda \zeta^X} m(dx),
}
By considering $\lambda = 0$, we find that if $\pr^x(\zeta^X < \infty) = 1$ for almost all $x \in D$, then
\formula[eq:invariant]{
 \ex^\diag f(Z) & = \ex^\diag f(X(0)) = \int_D f(x) m(dx) ,
}
a property that was stated as Proposition~\ref{prop:dynkin} in the previous section, and proved with a more direct approach when $X(t)$ is the Brownian motion.

%
%

\section{\uppercase{We'll meet again}}
\label{sec:fv}

We are now ready to prove Theorem~\ref{thm:main}. This is done below, after the Fleming--Viot system of two particles is introduced in a more careful way.

As in the previous section, we consider a self-dual Hunt process $X(t)$ on a state space $D$ with a finite reference measure $m$, and its independent copy $Y(t)$. Again with no loss of generality we assume that $m$ is a probability measure. In this section we consider the corresponding Fleming--Viot system of two particles: a bivariate process $(\bar X(t), \bar Y(t))$ which evolves just as $(X(t), Y(t))$, except that at an increasing sequence of \emph{branching times} $\tau_n$ the coordinate that is about to die, reenters the state space $D$ at the position of the other coordinate.

More precisely, the Fleming--Viot system $(\bar X(t), \bar Y(t))$ is constructed recursively. We let $\tau_0 = 0$, and we let $Z_0$ to be a random variable taking values in $D$, with distribution $m$. Once $\tau_{n - 1} \in [0, \infty)$ and $Z_{n - 1} \in D$ are given, and $\bar X(t)$ and $\bar Y(t)$ are defined on $[0, \tau_{n - 1})$, we proceed as follows. We sample an independent copy $(X_n(t), Y_n(t))$ of the bivariate process $(X(t), Y(t))$ started at the random position $(Z_{n - 1}, Z_{n - 1})$, and we denote the corresponding variables $\sigma$ and $Z$ by $\sigma_n$ and $Z_n$. That is,
\formula{
 \sigma_n & = \min\{\zeta^{X_n}, \zeta^{Y_n}\} , & Z_n & = \begin{cases} X_n(\zeta^{Y_n}) & \text{if $\zeta^{Y_n} < \zeta^{X_n}$,} \\ Y_n(\zeta^{X_n}) & \text{if $\zeta^{X_n} < \zeta^{Y_n}$.} \end{cases}
}
We define
\formula{
 \bar X(\tau_{n - 1} + t) & = X_n(t) , & \bar Y(\tau_{n - 1} + t) & = Y_n(t)
}
for $t \in [0, \sigma_n)$. Furthermore, we let
\formula{
 \tau_n & = \tau_{n - 1} + \sigma_n .
}
We denote the probability corresponding to the above construction by $\bar \pr^\diag$ to emphasise that the initial configuration of the particles is distributed according to the measure $\diag$ on the diagonal of $D \times D$ with marginals $m$.

Exactly the same construction can be carried out for an arbitrary distribution of the initial configuration of the two particles. If $\bar X(0) = x$ and $\bar Y(0) = y$ with probability one, we write $\bar \pr^{x, y}$ for the corresponding probability. If $\bar X(0)$ and $\bar Y(0)$ are drawn independently from distribution $m$, we denote the corresponding probability by $\bar \pr^{m \times m}$. More generally, for an arbitrary probability measure $\mu$ on $D \times D$ we write $\bar \pr^\mu = \int_{D \times D} \bar \pr^{x, y} \mu(dx, dy)$ for the probability corresponding to the system of particles with initial configuration $(\bar X(0), \bar Y(z))$ chosen randomly from distribution $\mu$. Clearly, for any event $E$ we have
\formula[eq:disintegration]{
 \bar \pr^\mu(E) & = \iint\limits_{D \times D} \bar \pr^{x, y}(E) \mu(dx, dy) .
}
Note that the process $(\bar X(t), \bar Y(t))$ is defined on $[0, \tau_\infty)$, where
\formula{
 \tau_\infty & = \lim_{n \to \infty} \tau_n \in (0, \infty] .
}
If $\tau_\infty < \infty$, we say that the Fleming--Viot particle system becomes \emph{extinct} (or suffers from a \emph{branching explosion}) at time $\tau_\infty$, and in this case we set $\bar X(t) = \bar Y(t) = \partial$ for $t \ge \tau_\infty$ to make the definition of $(\bar X(t), \bar Y(t))$ complete.

Below we restate, and then prove, Theorem~\ref{thm:main}. For convenience, we show items~(a) and~(b) of Theorem~\ref{thm:main} separately, as Corollary~\ref{cor:two} and Theorem~\ref{thm:two}, respectively.

\begin{theorem}
\label{thm:two}
Consider the Fleming--Viot system $(\bar X(t), \bar Y(t))$ of two particles evolving according to a self-dual Hunt process $X(t)$ on a state space $D$ with a finite reference measure $m$. For almost every $x \in D$ (with respect to $m$), if the initial configuration is $\bar X(0) = \bar Y(0) = x$, then the system never becomes extinct.
\end{theorem}

\begin{proof}
With no loss of generality we assume that $m$ is a probability measure.

By construction, up to the first branching time, the process $(\bar X(t), \bar Y(t))$ is a copy of the process $(X(t), Y(t))$ studied in the previous section. More precisely, for $t \in [0, \tau_1)$ the process $(\bar X(t), \bar Y(t)) = (X_1(t), Y_1(t))$ is a copy of the process $(X(t), Y(t))$ restricted to $t \in [0, \sigma)$. Note, however, that at time $\tau_1$ we have $\bar X(\tau_1) = \bar Y(\tau_1) = Z_1$, while one of the variables $X_1(\tau_1)$ and $Y_1(\tau_1)$ is equal to $\partial$.

By definition, $Z_1 = X_1(\tau_1)$ if $\tau_1 = \zeta^{Y_1}$, and $Z_1 = Y_1(\tau_1)$ if $\tau_1 = \zeta^{X_1}$. This means that $Z_1$ is defined in terms of $X_1(t)$ and $Y_1(t)$ in the same way as the random variable $Z$ was constructed in~\eqref{eq:z} using $X(t)$ and $Y(t)$.

Again by construction, the process $(\bar X(t), \bar Y(t))$ has the strong Markov property at time $\tau_1$: the evolution after time $\tau_1$ is independent of the past, conditionally on the present value of $Z_1 = \bar X(\tau_1) = \bar Y(\tau_1)$. We fix $\lambda > 0$, and we write
\formula{
 f(x) & = \bar \ex^{x, x} e^{-2 \lambda \tau_\infty} .
}
By the strong Markov property and~\eqref{eq:key}, we have
\formula{
 \int_D f(x) m(dx) & = \bar \ex^\diag e^{-2 \lambda \tau_\infty} \\
 & = \bar \ex^\diag \bigl(\ind_{\{\tau_1 < \infty\}} e^{-2 \lambda \tau_1} e^{-2 \lambda (\tau_\infty - \tau_1)}\bigr) \displaybreak[0] \\
 & = \bar \ex^\diag \bigl(\ind_{\{\tau_1 < \infty\}} e^{-2 \lambda \tau_1} \bar \ex^\diag (e^{-2 \lambda (\tau_\infty - \tau_1)} \vert \mathscr F_{\tau_1})\bigr) \displaybreak[0] \\
 & = \bar \ex^\diag \bigl(\ind_{\{\tau_1 < \infty\}} e^{-2 \lambda \tau_1} f(Z_1) \bigr) \displaybreak[0] \\
 & = \ex^\diag \bigl(\ind_{\{\sigma < \infty\}} e^{-2 \lambda \sigma} f(Z) \bigr) \\
 & = \int_D f(x) \ex^x e^{-\lambda \zeta^X} m(dx) .
}
Thus,
\formula{
 \int_D f(x) (1 - \ex^x e^{-\lambda \zeta^X}) m(dx) & = 0 .
}
However, $\ex^x e^{-\lambda \zeta^X} < 1$ for every $x \in D$. It follows that $f(x) = 0$ for almost every $x \in D$, that is, $\tau_\infty = \infty$ with probability $\bar\pr^{x, x}$ one for almost every $x \in D$.
\end{proof}

Claim~(b) in Theorem~\ref{thm:main} is thus proved. In order to extend this result to almost all initial configurations $(\bar X(0), \bar Y(0)) = (x, y)$ with respect to the product measure $m \times m$, as in claim~(a), we need one more step.

\begin{corollary}
\label{cor:two}
The result of Theorem~\ref{thm:two} remains true for almost every initial configuration $\bar X(0) = x$, $\bar Y(0) = y$ (with respect to the product measure $m \times m$), and also when the distribution of $(\bar X(0), \bar Y(0))$ is absolutely continuous with respect to $m \times m$.
\end{corollary}

\begin{proof}
Again, with no loss of generality we assume that $m$ is a probability measure. Suppose that the initial position of $(\bar X(t), \bar Y(t))$ is chosen randomly according to the product measure $m \times m$, and recall that we denote the corresponding probability by $\bar \pr^{m \times m}$. As in the proof of Theorem~\ref{thm:two}, we observe that up to the first branching time $\tau_1$, we have $\bar X(t) = X_1(t)$ and $\bar Y(t) = Y_1(t)$, where $X_1(t)$ and $Y_1(t)$ are independent copies of the underlying Hunt process $X(t)$, and both are started independently at a random position in $D$ chosen according to the measure $m$. Furthermore, we have $Z_1 = X_1(\tau_1)$ if $\tau_1 = \zeta^{Y_1}$ and $Z_1 = Y_1(\tau_1)$ if $\tau_1 = \zeta^{X_1}$. It follows that for every Borel set $A \subseteq D$,
\formula{
 \bar \pr^{m \times m}(Z_1 \in A) & = \bar \pr^{m \times m}(X_1(\zeta^{Y_1}) \in A) + \bar \pr^{m \times m}(Y_1(\zeta^{X_1}) \in A) \\
 & = 2 \bar \pr^{m \times m}(X_1(\zeta^{Y_1}) \in A)
}
(recall that $X_1(\zeta^{Y_1}) = \partial$ if $\zeta^{Y_1} \ge \zeta^{X_1}$). By independence,
\formula{
 \bar \pr^{m \times m}(Z_1 \in A) & = 2 \int_{[0, \infty)} \bar \pr^{m \times m}(X_1(t) \in A) \bar \pr^{m \times m}(\zeta^{Y_1} \in dt) .
}
If $\pr^m$ corresponds to the process $X(t)$ started at a random position $X(0)$ with distribution $m$, then we obtain
\formula{
 \bar \pr^{m \times m}(Z_1 \in A) & = 2 \int_{[0, \infty)} \pr^m(X(t) \in A) \pr^m(\zeta^X \in dt) .
}
However,
\formula{
 \pr^m(X(t) \in A) & = \int_D \pr^x(X(t) \in A) m(dx) = \langle \ind, P_t \ind_A \rangle ,
}
and since $P_t$ is self-adjoint, we have
\formula{
 \pr^m(X(t) \in A) & = \langle P_t \ind, \ind_A \rangle \le \langle \ind, \ind_A \rangle = m(A) .
}
Therefore,
\formula{
 \bar \pr^{m \times m}(Z_1 \in A) & \le 2 \int_{[0, \infty)} m(A) \pr^m(\zeta^X \in dt) = 2 m(A) .
}
In particular, the distribution of $Z_1$ is absolutely continuous with respect to $m$.

By the strong Markov property, the shifted process $(\bar X(\tau_1 + t), \bar Y(\tau_1 + t))$ is the same Fleming--Viot particle system, with initial configuration $(Z_1, Z_1)$. Above we proved that the distribution of $Z_1$ under $\bar \pr^{m \times m}$ is absolutely continuous with respect to $m$. Thus, using the strong Markov property and Theorem~\ref{thm:two}, we find that
\formula{
 \bar \pr^{m \times m}(\tau_\infty < \infty) & = \bar \ex^{m \times m}\bigl(\ind_{\{\tau_1 < \infty\}} \bar \pr^{m \times m}(\tau_\infty < \infty) \big| \mathscr F_{\tau_1}\bigr) \\
 & = \bar \ex^{m \times m}\bigl(\ind_{\{\tau_1 < \infty\}} \bar \pr^{Z_1, Z_1}(\tau_\infty < \infty)\bigr) \\
 & = \int_D \bar \pr^{x, x}(\tau_\infty < \infty) \bar \pr^{m \times m}(Z_1 \in dx) = 0 ,
}
that is, the system never becomes extinct. By~\eqref{eq:disintegration} (with $\mu = m \times m$), we conclude that $\tau_\infty = \infty$ with probability $\bar \pr^{x, y}$ one for almost every pair $x, y$.

The latter assertion of the lemma follows immediately from the former by~\eqref{eq:disintegration} and Fubini. Alternatively, one can repeat the above argument with the initial configuration distributed according to a given absolutely continuous distribution rather than $m \times m$.
\end{proof}

Of course, this proves claim~(a) of Theorem~\ref{thm:main}, and so the proof of our main result is complete.

\begin{proof}[Proof of Corollary~\ref{cor:all}]
Suppose that $\bar X(t)$ and $\bar Y(t)$ are started at fixed points $x$ and $y$, respectively. Then the processes $\bar X(t)$ and $\bar Y(t)$ are independent up to the first branching time $\tau_1$. By the same argument as in the proof of Corollary~\ref{cor:two}, the distribution of $Z_1$ is a mixture of the one-dimensional distributions of the underlying Hunt process $X(t)$, and hence, by assumption, it is absolutely continuous with respect to $m$. The remaining part of the proof is exactly the same as in the proof of Corollary~\ref{cor:two}.
\end{proof}

%
%

\section{\uppercase{Isn't this where we came in?}}
\label{sec:ergodic}

In this section we prove Theorem~\ref{thm:ergodic} and Corollary~\ref{cor:embedded}. We use the notation introduced in the previous section, and we begin by showing that $m$ is a stationary measure for the embedded Markov chain $Z_n$.

\begin{proof}[Proof of Corollary~\ref{cor:embedded}]
With no loss of generality we assume that $m$ is a probability measure, and we suppose that the initial configuration $(\bar X(0), \bar Y(0))$ has distribution $\delta$ on the diagonal of $D \times D$, as in the proof of Theorem~\ref{thm:two}. By assumption, $Z_0 = \bar X(0) = \bar Y(0)$ has distribution $m$, and $\pr^x(\zeta^X < \infty) = 1$ for almost every $x \in D$.

As it was observed in the proof of Theorem~\ref{thm:two}, the process $(\bar X(t), \bar Y(t))$ up to the first branching time $\tau_1$ is a copy of the bivariate process $(X(t), Y(t))$ up to time $\sigma = \min\{\zeta^X, \zeta^Y\}$, and hence the distribution of $Z_1$ under $\bar \pr^\delta$ is equal to the distribution of $Z$ under $\pr^\delta$. By~\eqref{eq:invariant}, the latter is equal to $m$. Therefore, $Z_1$ has distribution $m$.

Suppose that we already know that $Z_{n - 1}$ has distribution $m$. By the strong Markov property for the process $(\bar X(t), \bar Y(t))$ at time $\tau_{n - 1}$ (see~\cite{inw,meyer}), the shifted process $(\bar X(\tau_{n - 1} + t), \bar Y(\tau_{n - 1} + t))$ under $\bar \pr^\delta$ is just a copy of the original process $(\bar X(t), \bar Y(t))$ under $\bar \pr^\delta$. The result of the previous paragraph implies that $Z_n$ has distribution $m$, and additionally it proves the Markov property for the sequence $Z_0, Z_1, \ldots$ at time $n - 1$. Thus, by induction, the sequence $Z_0, Z_1, \ldots$ is indeed a Markov chain, and for every $n$ the random variable $Z_n$ has distribution $m$.
\end{proof}

For the construction of a stationary measure for the Fleming--Viot system $(\bar X(t), \bar Y(t))$ in Theorem~\ref{thm:ergodic}, we additionally denote
\formula{
 \mu(dx, dy) & = \|G\|^{-1} G(x, dy) m(dx)
}
whenever $\|G\|$ is finite; see~\eqref{eq:g} and~\eqref{eq:g:norm}. Note that in this case $\mu$ is a probability measure on $D \times D$. Our goal is to show that under probability $\bar \pr^\mu$, corresponding to the Fleming--Viot system of two particles with initial configuration $(\bar X(0), \bar Y(0))$ having distribution $\mu$, the system $(\bar X(t), \bar Y(t))$ is stationary.

The only essential property that we use here is that the process $(\bar X(t), \bar Y(t))$ is obtained by concatenating `excursions' $(X_k(t), Y_k(t))$, $t \in [0, \sigma_k)$, and that the initial configurations $(X_k(0), Y_k(0)) = (Z_{k - 1}, Z_{k - 1})$ of these excursions form a stationary Markov chain (by Corollary~\ref{cor:embedded}). While such a concatenation procedure seems to be rather standard (see, for example, \cite{inw,meyer}), the author failed to find a proper reference, and so we include full details.

We fix $\lambda > 0$ and two bounded nonnegative Borel functions $f, g$, and we denote
\formula{
 \ph(x, y) & = \bar \ex^{x, y} \int_0^\infty e^{-\lambda t} f(\bar X(t)) g(\bar Y(t)) dt .
}
Recall that $\delta$ is the measure on the diagonal of $D \times D$ with marginals $m$, and once again with no loss of generality we assume that $m$ is a probability measure. The following lemma is the key technical result in the proof of Theorem~\ref{thm:ergodic}

\begin{lemma}
\label{lem:stationary}
Suppose that $\|G\|$ is finite. With the above definitions, we have
\formula[eq:stationary]{
 \frac{1}{2} \iint\limits_{D \times D} \ph(x, y) G(x, dy) m(dx) & = \frac{1}{2 \lambda} \iint\limits_{D \times D} f(x) g(y) G(x, dy) m(dx) .
}
\end{lemma}

\begin{proof}
Recall that $G(x, dy) = u_0(x, dy)$, and that the process $(\bar X(t), \bar Y(t))$ up to the first branching time $\tau_1$ is a copy of the process $(X(t), Y(t))$ up to time $\sigma = \min\{\zeta^X, \zeta^Y\}$. Thus, formula~\eqref{eq:reduction:sigma} implies that
\formula{
 \frac{1}{2} \smash{\iint\limits_{D \times D}} \ph(x, y) G(x, dy) m(dx) & = \ex^\delta \biggl(\int_0^\sigma \ph(X(s), Y(s)) ds\biggr) \\
 & = \bar \ex^\delta \biggl(\int_0^{\tau_1} \ph(\bar X(s), \bar Y(s)) ds\biggr) .
}
We now show a variant of the resolvent equation, with one integral over $(0, \infty)$ and the other one over $(0, \tau_1)$. Using the Markov property, we find that
\formula{
 \hspace*{3em} & \hspace*{-3em} \frac{1}{2} \smash{\iint\limits_{D \times D}} \ph(x, y) G(x, dy) m(dx) \\
 & = \bar \ex^\delta \biggl(\int_0^{\tau_1} \bar \ex^{\bar X(s), \bar Y(s)} \biggl(\int_0^\infty e^{-\lambda t} f(\bar X(t)) g(\bar Y(t)) dt\biggr) ds\biggr) \displaybreak[0] \\
 & = \bar \ex^\delta \biggl(\int_0^{\tau_1} \bar \ex^\delta \biggl(\int_0^\infty e^{-\lambda t} f(\bar X(s + t)) g(\bar Y(s + t)) dt \bigg| \mathscr F_s \biggr) ds\biggr) \displaybreak[0] \\
 & = \bar \ex^\delta \biggl(\int_0^{\tau_1} \int_0^\infty e^{-\lambda t} f(\bar X(s + t)) g(\bar Y(s + t)) dt ds\biggr) \\
 & = \bar \ex^\delta \biggl(\int_0^{\tau_1} \int_s^\infty e^{\lambda s - \lambda t} f(\bar X(t)) g(\bar Y(t)) dt ds\biggr) .
}
By Fubini's theorem,
\formula*[eq:ergodic:0]{
 \hspace*{3em} & \hspace*{-3em} \frac{1}{2} \smash{\iint\limits_{D \times D}} \ph(x, y) G(x, dy) m(dx) \\
 & = \bar \ex^\delta \biggl(\int_0^\infty \int_0^{\min\{t, \tau_1\}} e^{\lambda s - \lambda t} f(\bar X(t)) g(\bar Y(t)) ds dt\biggr) \\
 & = \bar \ex^\delta \biggl(\int_0^{\tau_1} \frac{1 - e^{-\lambda t}}{\lambda} \, f(\bar X(t)) g(\bar Y(t)) dt\biggr) \\
 & \hspace*{3em} + \bar \ex^\delta \biggl(\int_{\tau_1}^\infty \frac{e^{\lambda \tau_1 - \lambda t} - e^{-\lambda t}}{\lambda} \, f(\bar X(t)) g(\bar Y(t)) dt\biggr) \\
 & = \frac{1}{\lambda} \, \bar \ex^\delta \biggl(\int_0^{\tau_1} f(\bar X(t)) g(\bar Y(t)) dt\biggr) \\
 & \hspace*{3em} - \frac{1}{\lambda} \, \bar \ex^\delta \biggl(\int_0^{\tau_1} e^{-\lambda t} f(\bar X(t)) g(\bar Y(t)) dt\biggr) \\
 & \hspace*{3em} + \frac{1}{\lambda} \, \bar \ex^\delta \biggl(\int_{\tau_1}^\infty e^{\lambda \tau_1 - \lambda t} f(\bar X(t)) g(\bar Y(t)) dt\biggr) \\
 & \hspace*{3em} - \frac{1}{\lambda} \, \bar \ex^\delta \biggl(\int_{\tau_1}^\infty e^{-\lambda t} f(\bar X(t)) g(\bar Y(t)) dt\biggr) .
}
We study each of the terms on the right-hand side. For the first one, by equality of $(\bar X(t), \bar Y(t))$ and $(X(t), Y(t))$ up to the first branching time, and by~\eqref{eq:reduction:sigma} with $\lambda = 0$, we obtain
\formula*[eq:ergodic:1]{
 \bar \ex^\delta \biggl(\int_0^{\tau_1} f(\bar X(t)) g(\bar Y(t)) dt\biggr) & = \ex^\delta \biggl(\int_0^\sigma f(X(t)) g(Y(t)) dt\biggr) \\
 & = \frac{1}{2} \, \langle f, U_0 g \rangle \\
 & = \frac{1}{2} \iint_{D \times D} f(x) g(y) G(x, dy) m(dx) .
}
In order to transform the third term, we apply the strong Markov property:
\formula{
 \hspace*{3em} & \hspace*{-3em} \bar \ex^\delta \biggl(\int_{\tau_1}^\infty e^{\lambda \tau_1 - \lambda t} f(\bar X(t)) g(\bar Y(t)) dt\biggr) \\
 & = \bar \ex^\delta \biggl(\ind_{\{\tau_1 < \infty\}} \int_0^\infty e^{-\lambda t} f(\bar X(\tau_1 + t)) g(\bar Y(\tau_1 + t)) dt\biggr) \displaybreak[0] \\
 & = \bar \ex^\delta \biggl(\ind_{\{\tau_1 < \infty\}} \bar \ex^\delta \biggl(\int_0^\infty e^{-\lambda t} f(\bar X(\tau_1 + t)) g(\bar Y(\tau_1 + t)) dt \bigg| \mathscr F_{\tau_1}\biggr) \biggr) \displaybreak[0] \\
 & = \bar \ex^\delta \biggl(\ind_{\{\tau_1 < \infty\}} \bar \ex^{Z_1, Z_1} \biggl(\int_0^\infty e^{-\lambda t} f(\bar X(t)) g(\bar Y(t)) dt\biggr) \biggr) \\
 & = \bar \ex^\delta \bigl(\ind_{\{\tau_1 < \infty\}} \ph(Z_1, Z_1)\bigr) .
}
Since $\|G\|$ is finite, we have $\bar \pr^{x, x}(\tau_1 < \infty) = \pr^{x, x}(\sigma < \infty) = 1$ for almost every $x \in D$, and hence $\bar \pr^\delta(\tau_1 < \infty) = 1$. By Corollary~\ref{cor:embedded}, the distribution of $Z_1$ under $\bar \pr^\delta$ is equal to $m$. Thus,
\formula{
 \bar \ex^\delta \biggl(\int_{\tau_1}^\infty e^{\lambda \tau_1 - \lambda t} f(\bar X(t)) g(\bar Y(t)) dt\biggr) & = \int_D \ph(x, x) m(dx) .
}
By the definition of $\ph$, we have
\formula*[eq:ergodic:3]{
 \hspace*{3em} & \hspace*{-3em} \bar \ex^\delta \biggl(\int_{\tau_1}^\infty e^{\lambda \tau_1 - \lambda t} f(\bar X(t)) g(\bar Y(t)) dt\biggr) \\
 & = \int_D \bar \ex^{x, x} \biggl(\int_0^\infty e^{-\lambda t} f(\bar X(t)) g(\bar Y(t)) dt\biggr) m(dx) \\
 & = \bar \ex^\delta \biggl(\int_0^\infty e^{-\lambda t} f(\bar X(t)) g(\bar Y(t)) dt\biggr) .
}
Combining~\eqref{eq:ergodic:1} and~\eqref{eq:ergodic:3} with~\eqref{eq:ergodic:0}, we find that
\formula{
 \frac{1}{2} \iint\limits_{D \times D} \ph(x, y) G(x, dy) m(dx) & = \frac{1}{2 \lambda} \iint\limits_{D \times D} f(x) g(y) G(x, dy) m(dx) \\
 & \hspace*{3em} - \frac{1}{\lambda} \, \bar \ex^\delta \biggl(\int_0^{\tau_1} e^{-\lambda t} f(\bar X(t)) g(\bar Y(t)) dt\biggr) \\
 & \hspace*{3em} + \frac{1}{\lambda} \, \bar \ex^\delta \biggl(\int_0^\infty e^{-\lambda t} f(\bar X(t)) g(\bar Y(t)) dt\biggr) \\
 & \hspace*{3em} - \frac{1}{\lambda} \, \bar \ex^\delta \biggl(\int_{\tau_1}^\infty e^{-\lambda t} f(\bar X(t)) g(\bar Y(t)) dt\biggr) .
}
The last three terms on the right-hand side cancel out, and the desired result follows.
\end{proof}

\begin{proof}[Proof of Theorem~\ref{thm:ergodic}]
We divide the argument into three parts.

\emph{Step 1.} Suppose that $\|G\|$ is finite. Using the definition of $\bar \pr^\mu$ and $\ph$, formula~\eqref{eq:stationary} can be rewritten as
\formula{
 \bar \ex^\mu \int_0^\infty e^{-\lambda t} f(\bar X(t)) g(\bar Y(t)) dt & = \frac{1}{\lambda} \, \bar \ex^\mu \bigl(f(\bar X(0)) g(\bar Y(0))\bigr) .
}
Fubini's theorem implies that if $\Phi(t) = \bar \ex^\mu \bigl(f(\bar X(t)) g(\bar Y(t))\bigr)$, then
\formula{
 \int_0^\infty e^{-\lambda t} \Phi(t) dt & = \frac{\Phi(0)}{\lambda} \, ,
}
that is, the Laplace transform of $\Phi$ is given by $\Phi(0) / \lambda$. It follows that $\Phi(t) = \Phi(0)$ for almost every $t \in [0, \infty)$. Assume that $f$ and $g$ are additionally continuous. Then $\Phi$ is right-continuous, and hence we simply have $\Phi(t) = \Phi(0)$ for every $t \in [0, \infty)$. In other words,
\formula{
 \bar \ex^\mu \bigl(f(\bar X(t)) g(\bar Y(t))\bigr) & = \bar \ex^\mu \bigl(f(\bar X(0)) g(\bar Y(0))\bigr)
}
for every $t \ge 0$ and every bounded continuous nonnegative functions $f$ and $g$. By a density argument, the distributions of $(\bar X(t), \bar Y(t))$ and $(\bar X(0), \bar Y(0))$ under $\bar \pr^\mu$ are equal, that is, $(\bar X(t), \bar Y(t))$ is indeed a stationary process under $\pr^\mu$. The first assertion of Theorem~\ref{thm:ergodic} is proved when $\|G\|$ is finite.

\emph{Step 2.} Extension to the case when $G(x, dy) m(dx)$ is a $\sigma$-finite measure is immediate, except that we need to work with the infinite measure $\mu(dx, dy) = G(x, dy) m(dx)$ (without the normalisation constant $\|G\|^{-1}$). We leave it to the interested reader to verify that the above proof carries over to this setting.

\emph{Step 3.} The second assertion of Theorem~\ref{thm:ergodic} follows now from Birkhoff's ergodic theorem for Markov processes. Indeed: suppose that $\ph$ is a nonnegative Borel function on $D \times D$. By the ergodic theorem given in Corollary~25.9 in~\cite{kallenberg}, the limit
\formula{
 M & = \lim_{T \to \infty} \frac{1}{T} \int_0^T \ph(\bar X(t), \bar Y(t)) dt
}
exists with probability $\bar \pr^\mu$ one, and if $\mathscr I$ denotes the $\sigma$-algebra of all events which are invariant under time shifts, then
\formula{
 M & = \bar \ex^\mu\bigl(\ph(\bar X(0), \bar Y(0)) \big\vert \mathscr I\bigr) .
}
Below we prove that irreducibility of $X(t)$ implies that $\mathscr I$ is trivial: it only contains events of probability $\bar \pr^\mu$ zero or one. This implies that
\formula{
 M & = \bar \ex^\mu \ph(\bar X(0), \bar Y(0)) = \iint\limits_{D \times D} \ph(x, y) \mu(dx, dy) = \frac{1}{\|G\|} \iint\limits_{D \times D} \ph(x, y) G(x, dy) m(dx)
}
with probability $\bar \pr^\mu$ one, completing the proof of the theorem.

It remains to show that $\mathscr I$ is trivial. This follows by a standard argument, which is however difficult to find in literature, and so we provide full details. Recall that we assume irreducibility of $X(t)$: if $m(A) > 0$ and $t > 0$, then $\pr^x(X_t \in A) > 0$ for almost every $x \in D$, and our goal is to prove that if $I$ is an invariant event for $(\bar X(t), \bar Y(t))$, in the sense that the time-shifts leave $I$ unchanged, then $\bar \pr^\mu(I)$ is either zero or one.

By Lemma~1 in~\cite{fukushima}, for every invariant event $I$ we have
\formula{
 \bar \pr^\mu\bigl(\ind_I = \ind_B(\bar X(0), \bar Y(0))\bigr) & = 1
}
for some invariant Borel set $B \subseteq D \times D$. That is, for every $t > 0$ we have
\formula{
 \bar \pr^{x, y}((\bar X(t), \bar Y(t)) \in B) & = \ind_B(x, y)
}
for almost every $(x, y) \in D \times D$ with respect to the measure $\mu$. Since the product measure $m \times m$ is absolutely continuous with respect to $\mu(dx, dy) = \|G\|^{-1} G(x, dy) m(dx)$, the above property also holds for almost every $(x, y) \in D \times D$ with respect to $m \times m$. It follows that for every $t > 0$,
\formula{
 \pr^{x, y}((X(t), Y(t)) \in B) & \le \ind_B(x, y)
}
for almost every $(x, y) \in D \times D$ (with respect to $m \times m$). On the other hand, irreducibility of $X(t)$ and $Y(t)$ and independence of these processes imply that if $m \times m(B) > 0$, then for almost every $(x, y) \in D \times D$ we have
\formula{
 \pr^{x, y}((X(t), Y(t)) \in B) & > 0 .
}
Thus, if $m \times m(B) > 0$, then $\ind_B(x, y) > 0$ for almost every $(x, y) \in D \times D$, that is, $B$ is of full measure $m \times m$. We conclude that either $B$ or its complement has zero measure $m \times m$. In the former case $\bar \pr^\mu(I) = 0$, while in the latter $\bar \pr^\mu(I) = 1$, and so $\mathscr I$ is indeed trivial.
\end{proof}

%
%

\bigskip
\subsection*{Acknowledgements}

I am immensely grateful to Krzysztof Burdzy and Jim Pitman for introduction to the subject and inspiration, and to the anonymous referee for insightful and exceptionally helpful comments.

%
%

%
%

\end{document}